\documentclass[12pt]{amsart}
\usepackage{amsmath,amssymb,amsbsy,amsfonts,latexsym,amsopn,amstext,
                                               amsxtra,euscript,amscd}
\usepackage{url}
\usepackage[colorlinks,linkcolor=blue,anchorcolor=blue,citecolor=blue]{hyperref}
\usepackage{color}
\usepackage{float}
  \newtheorem{thm}{Theorem}
  
  \newtheorem{lem}[thm]{Lemma}
  
  \theoremstyle{definition}
  
  \theoremstyle{remark}

\begin{document}

\def\squareforqed{\hbox{\rlap{$\sqcap$}$\sqcup$}}
\def\qed{\ifmmode\squareforqed\else{\unskip\nobreak\hfil
\penalty50\hskip1em\null\nobreak\hfil\squareforqed
\parfillskip=0pt\finalhyphendemerits=0\endgraf}\fi}

\def\cA{{\mathcal A}}
\def\cB{{\mathcal B}}
\def\cC{{\mathcal C}}
\def\cD{{\mathcal D}}
\def\cE{{\mathcal E}}
\def\cF{{\mathcal F}}
\def\cG{{\mathcal G}}
\def\cH{{\mathcal H}}
\def\cI{{\mathcal I}}
\def\cJ{{\mathcal J}}
\def\cK{{\mathcal K}}
\def\cL{{\mathcal L}}
\def\cM{{\mathcal M}}
\def\cN{{\mathcal N}}
\def\cO{{\mathcal O}}
\def\cP{{\mathcal P}}
\def\cQ{{\mathcal Q}}
\def\cR{{\mathcal R}}
\def\cS{{\mathcal S}}
\def\cT{{\mathcal T}}
\def\cU{{\mathcal U}}
\def\cV{{\mathcal V}}
\def\cW{{\mathcal W}}
\def\cX{{\mathcal X}}
\def\cY{{\mathcal Y}}
\def\cZ{{\mathcal Z}}

\def\fH{{\mathfrak H}}
\def\fR{{\mathfrak R}}
\def\fM{{\mathfrak M}}

\def \C {{\mathbb C}}
\def \F {{\mathbb F}}
\def \L {{\mathbb L}}
\def \K {{\mathbb K}}
\def \Q {{\mathbb Q}}
\def \Z {{\mathbb Z}}

\def\\{\cr}
\def\({\left(}
\def\){\right)}
\def\[{\left[}
\def\]{\right]}
\def\fl#1{\left\lfloor#1\right\rfloor}
\def\cl#1{\left\lceil#1\right\rceil}

\def \lcm{{\mathrm {lcm}}}
\def \rad{{\mathrm {rad}}}
\def \ord{{\mathrm {ord}}}
\def \llog{{\mathrm{llog~}}}
\def \Li{{\mathrm {Li}}}
\def\e{\mathbf{e}}
\def\em{\e_m}
\def\el{\e_\ell}
\def\Res{\mathrm{Res}}

\def\ellax{\vec{a} \cdot \vec{x}}
\def\ellay{\vec{a} \cdot \vec{y}}

\def\mand{\qquad \text{and} \qquad}
\renewcommand{\vec}[1]{\mathbf{#1}}

\newcommand{\comm}[1]{\marginpar{%
\vskip-\baselineskip 
\raggedright\footnotesize
\itshape\hrule\smallskip#1\par\smallskip\hrule}}

\title{Counting Irreducible Binomials over Finite Fields}

 \author[R. Heyman] {Randell Heyman}

\address{Department of Pure Mathematics, University of New South Wales,
Sydney, NSW 2052, Australia}
\email{randell@unsw.edu.au}

 \author[I. E. Shparlinski] {Igor E. Shparlinski}

\address{Department of Pure Mathematics, University of New South Wales,
Sydney, NSW 2052, Australia}
\email{igor.shparlinski@unsw.edu.au}


\begin{abstract} We consider various counting questions
for irreducible binomials over finite fields. We use various
results from analytic number theory to investigate these
questions.
\end{abstract}



\maketitle

\section{Introduction}

\subsection{Background}

It is reasonably easy to obtain an asymptotic formula for the
total number of  irreducible polynomials
over the finite field $\F_q$ of $q$ elements, see~\cite[Theorem~3.25]{LiNi}.

Studying irreducible  polynomials with some prescribed coefficients
is much more difficult, yet remarkable progress has also been achieved
in this direction, see~\cite{Cohen, Hucz, Poll} and references therein.

Here we consider a special case of this problem
and investigate some counting questions concerning irreducible binomials
over the finite field $\F_q$ of $q$ elements.  More precisely, for an integer $t$
and a  prime power $q$,
 let $N_t(q)$ be the number of irreducible binomials
 over $\F_q$ of the form $X^t-a \in  \F_q[X]$.

We use a well known characterisation of irreducible binomials $X^t-a$
over  $\F_q$ of $q$ elements to count the total number of such
binomials on average over $q$ or $t$. In fact,  we consider several
natural regimes, for example, when  $t$ is fixed and $q$ varies
or when both vary in certain ranges $t \le T$ and $q\le Q$.
There has always been very active
interest in  binomials, see~\cite[Notes to Chapter~3]{LiNi} for a survey of classical
results.
Furthermore, irreducible binomials have been used in~\cite{Shoup}
as building blocks for constructing other irreducible polynomials over
finite fields, and in ~\cite{BMGVBO} for characterising the irreducible factors of $x^n-1$ (see also~\cite{ABK, MaZi} and references therein
for more recent applications).
However,  the natural  question of investigating the behaviour of $N_t(q)$
has never been addressed in the literature.

Our methods rely on several classical and modern results of analytic number
theory; in particular the distribution of primes in arithmetic progressions.

\subsection{Notation}
\label{sec:not}

As usual, let $\omega(s)$, $\pi(s)$, $\varphi(s)$, $\Lambda(s)$ and $\zeta(s)$  denote the number of distinct prime factors of $s$, the number of prime numbers less than or equal to $s$, the Euler totient function, the von Mangoldt function and the
Riemann-zeta function evaluated at $s$ respectively.

For positive integers $Q$ and $s$ we denote the number of primes in arithmetic progression by
 $$\pi(Q;s,a)=\sum_{\substack{p \le Q\\p \equiv a \pmod s}}1.$$
We also denote
$$\psi(Q;s,a)=\sum_{\substack{p \le Q\\p \equiv a \pmod s}}\Lambda(p).$$
The letter $p$ always denotes a prime number whilst the letter $q$ always denotes
a prime power.

We recall that the notation $f(x) = O(g(x))$  or $f(x) \ll g(x)$ is
equivalent to the assertion that there exists a constant $c>0$ (which may depend on the real parameter $\varepsilon> 0$) such that $|f(x)|\le c|g(x)|$ for all $x$. The notation
$f(x)=o(g(x))$ is equivalent to the assertion that $$\lim_{x \rightarrow \infty}\frac{f(x)}{g(x)}=0.$$
The notation $f(x) \sim g(x)$ is equivalent to the assertion that $$\lim_{x \rightarrow \infty} \frac{f(x)}{g(x)}=1.$$

We define
$\log x$ as $\log x=\max\{\ln x, 2\}$ where  $\ln x$ is the natural logarithm,
Furthermore,  for an integer $k\ge 2$, we define
recursively $\log_k x=\log(\log_{k-1}x)$.

Finally, we use $\Sigma^\sharp$ to indicate that the summation is
only over squarefree arguments in the range of summation.

\subsection{Main results}

We denote the radical of an integer $t\ne 0$, the largest square-free number that divides $t$, by $\rad(t)$. It is also convenient to define
$$\rad_4(t)=\begin{cases}\rad(t)&\mbox{if } 4  \nmid t,\\
2\rad(t)&\mbox{otherwise}.
\end{cases}$$

We start with an upper bound
on the average value of $N_t(q)$ for a fixed $t$ averaged over $q\le Q$.

\begin{thm}
\label{thm:UppBound q}
For any fixed $\varepsilon>0$ uniformly over real $Q$ and  positive integers  $t$
with $\rad_4(t) \le Q^{1-\varepsilon}$, we have
$$\sum_{q \le Q}N_t(q)\le (1 + o(1)) \frac{Q^2}{\rad_4(t)\log (Q/\rad_4(t))}$$
as $Q\to \infty$.
\end{thm}

We also present the following lower bound
(which has $\varphi(\rad(t))^2$ instead of the expected $\varphi(\rad(t))$).

\begin{thm}
\label{thm:LowBound q}
There exists an absolute constant $L > 0$ such that  uniformly over
 real $Q$ and  positive integers  $t$
with $Q \ge t^L$ we have
$$\sum_{q \le Q}N_t(q)\gg \frac{Q^2}{\varphi(\rad(t))^2(\log Q)^2}.$$
\end{thm}

We also investigate $N_t(q)$ for a fixed $q$ averaged  over $t\le T$.

\begin{thm}
\label{thm:UppBound t}
For any fixed positive $A$ and $\varepsilon$
and a sufficiently large real $q$ and  $T$ with
$$
T \ge \(\log (q-1)\)^{(1+ \varepsilon) A \log_3 q/\log_4 q}
$$
we have
$$
\sum_{t \le T}N_t(q)\le (q-1)T/(\log T)^A.
$$
\end{thm}

Finally, we obtain an asymptotic formula for the double
average of $N_t(q)$  over $q\le Q$ and squarefree $t\le T$ in a rather wide
range of parameters $Q$ and $T$. With more work similar results
can also be obtained for the average value of $N_t(q)$  over all
integers $t \le T$. However to exhibit the ideas and simplify the
exposition, we limit ourselves to this special case, in particular
we recall our notation $\Sigma^\sharp$  from Section~\ref{sec:not}.

 \begin{thm}
\label{thm:Asymp q t}
For  any fixed $\varepsilon> 0$ and any
$$
T \le Q^{1/2}/(\log Q)^{5/2+\varepsilon}
$$
we have
$$
\sum_{t \le T}\hskip-18 pt{\phantom{\sum}}^\sharp\,\sum_{q \le Q} N_t(q)
= (1+o(1))\frac{Q^2\log T}{2\zeta(2)\log Q},
$$
as $T\to \infty$.
\end{thm}

It seems difficult to obtain the asymptotic formula of Theorem~\ref{thm:Asymp q t}
for larger values of $T$ (even under the Generalised Riemann Hypothesis).
However, here we show that a result of Mikawa~\cite{Mik} implies a lower bound
of right order of magnitude for values of $T$ of order that may exceed $Q^{1/2}$.

 \begin{thm}
\label{thm:Lower q t}
For  any fixed  $\beta < 17/32$ and $T  \le  Q^\beta$, we have
$$
\sum_{T \le t \le 2T}\hskip-18 pt{\phantom{\sum}}^\sharp\,\sum_{q \le Q} N_t(q)
\gg \frac{Q^2}{\log Q},
$$
\end{thm}

We note that Theorem~\ref{thm:Lower q t} means that for a positive
proportion of fields $\F_q$ with $q \le Q$  there is a positive proportion of irreducible
binomials whose degrees do not exceed $Q^{\beta}$.

\section{Preparations}

\subsection{Characterisation of irreducible binomials}
Let $\ord_q a$ denote the multiplicative order of $a \in \F_q^*$.

Our main tool is the following characterisation of irreducible binomials
(see~\cite[Theorem~3.75]{LiNi}).

\begin{lem}
\label{lem:Irr Bin}
Let $t \ge 2$ be an integer and $a \in \F_q^*$. Then the binomial $x^t-a$ is irreducible in $\F_q[x]$ if and only if the following three conditions are satisfied:
\begin{enumerate}
\item $\rad(t) \mid \ord_q a$,
\item $ \gcd\(t,(q-1)/\ord_q a\)=1$,
\item if $4 \mid t$ then $q \equiv 1 \pmod 4$.
\end{enumerate}
\end{lem}

\begin{lem}\label{lem:ntq}
Suppose that $q$ is a prime power.
Then
$$N_t(q)=\begin{cases} \displaystyle{\frac{\varphi(t)}{t}(q-1)}, &
 \text{if } \rad_4(t)\mid (q-1) , \\
0,& \text{otherwise}.
\end{cases}$$
\end{lem}

\begin{proof} We can assume that $\rad_4(t)\mid (q-1)$(or equivalently
$\rad(t)\mid (q-1)$ and if $4\mid t$ then $q \equiv 1 \pmod 4$), as in
the opposite case the result is follows immediately from Lemma~\ref{lem:Irr Bin}.

Furthermore, from Lemma~\ref{lem:Irr Bin} we see that
$$N_t(q)=\sum_{\substack{a \in \F_q^*\\ \rad(t)\mid \ord_qa \\ \gcd(t,(q-1)/\ord_qa)=1}}1.$$
Since $\F_q^*$ is a cyclic group,  there are $\varphi(\ord_qa)$ elements of $\F_q^*$
that have order equal to $\ord_qa$. Hence, we obtain
$$
N_t(q)=\sum_{\substack{ j \mid (q-1) \\ \rad(t)\mid j  \\ \gcd(t,(q-1)/j)=1}} \varphi(j).
$$
We now write $q-1=RS$, where $R$ is the largest divisor of $q-1$ with $\gcd(R, \rad(t)) =1$
(thus all prime divisors of  $S$  also divide $t$).
Now, for every integer $j\mid (q-1)$ the conditions $\rad(t)\mid j$ and $\gcd(t,(q-1)/j)=1$ mean that
$j=Sd$  for some $d\mid R$. Since $\gcd(S,R)=1$, we have
$$
N_t(q)=\sum_{d\mid R} \varphi(Sd)=\varphi(S) \sum_{d\mid R} \varphi(d)=\varphi(S)R=\frac{\varphi(t)}{t}SR=\frac{\varphi(t)}{t}(q-1),
$$
which concludes the proof.
\end{proof}

\subsection{Analytic number theory background}

We recall a quantitative version of the Linnik  theorem,
see~\cite[Corollary~18.8]{IwKow}, which is slightly  stronger
than the form which is usually used.

\begin{lem}\label{lem:Linnik}
There is an absolute constant $L$ such that if
a positive integer $k$ is sufficiently large and $Q\ge k^L$, then
uniformly over all  integers $a$ with $\gcd(k,a)=1$ we have
$$\psi(Q;k,a) \gg \frac{Q}{\varphi(k)\sqrt{k}}.$$
\end{lem}

On average over $k$ we have a much more precise
result given by the  {\it Bombieri--Vinogradov theorem\/}
which we present in the form that follows from
the work of  Dress,  Iwaniec, and  Tenenbaum~\cite{DIT}
combined with the method of Vaughan~\cite{Vau}:

\begin{lem}\label{lem:Bomb-Vin}
For any $A>0$, $\alpha > 3/2$ and  $T \le Q$ we
have $$
\sum_{t \le T} \max_{\gcd(a,t)=1} \max_{R \le Q}
\left|\pi(R;t,a) - \frac{\pi(R)}{\varphi(t)}\right| \le
Q (\log Q)^{-A} + Q^{1/2} T (\log Q)^{\alpha}.
$$
\end{lem}

The following result follows immediately from much
more general estimates of Mikawa~\cite[Bounds~(4) and~(5)]{Mik}.

\begin{lem}
\label{lem:Mikawa}
For any fixed $\beta < 17/32$, $u  \le  z^\beta$ and for
all but $o(u)$ integers $k \in [u, 2u]$ we have
$$
\pi(2z;k,1) - \pi(z;k,1) \gg  \frac{z}{\varphi(k)\log z} .
$$
\end{lem}

We also have a bound on the number
$\rho_T(n)$ of integers $t\le T$ with $\rad(t) \mid n$,
which is due to Grigoriev and Tenenbaum~\cite[Theorem~2.1]{GrigTen}.
We note that~\cite[Theorem~2.1]{GrigTen} is formulated as a bound on
the number of divisors $t \mid n$ with $t \le T$.
However a direct examination of the argument reveals that it actually
provides an estimate for the above function $\rho_T(n)$. In fact we present
it in simpler form given by~\cite[Corollary~2.3]{GrigTen}

\begin{lem}\label{lem:SmallDiv} For any fixed positive $A$ and $\varepsilon$
and a sufficiently large positive integer $n$ and  a real $T$ with
$$
T \ge \(\log n\)^{(1+ \varepsilon) A \log_3 n/\log_4 n}
$$
we have $\rho_T(n)\le T/(\log T)^A$.
\end{lem}

\section{Proofs of Main Results}

\subsection{Proof of Theorem~\ref{thm:UppBound q}}

For the case where $4\nmid t$ we   denote $s = \rad(t)$. Using Lemma~\ref{lem:ntq} we have
\begin{equation}
\label{eq q-1 q}
\sum_{q \le Q}N_t(q)=\frac{\varphi(t)}{t}\sum_{\substack{q \le Q \\ s \mid (q-1)}}(q-1)
= \frac{\varphi(t)}{t}\sum_{\substack{q \le Q\\ s\mid (q-1)}}q + O(Q/s).
\end{equation}

So, with
$$
\ell  = \fl{\frac{\log Q}{\log 2}} \mand \lambda = 
2\varepsilon^{-1},
$$
we have
\begin{equation}
\label{eq:prime pow}
\sum_{\substack{q \le Q\\ s\mid (q-1)}}q  =
\sum_{\substack{p \le Q\\ s\mid (p-1)}}p + \sum_{2 \le r \le \ell}
\sum_{\substack{p^r\le Q\\ s\mid (p^r-1)}}p^r.
\end{equation}
Using the Brun-Titchmarsh bound, see~\cite[Theorem~6.6]{IwKow} and
partial summation we obtain
\begin{equation}
\label{eq:r=1}
\sum_{\substack{p \le Q\\ s\mid (p-1)}}p \le (1+o(1)) \frac{Q^2}{\varphi(s) \log (Q/s)},
\end{equation}
provided that $s/Q \to 0$.

We now estimate the contribution from other terms with $r \ge 2$.

The condition $s\mid p^r-1$ puts $p$ in at most $r^{\omega(s)}$ arithmetic progressions
modulo $s$. Extending the summation to all integers $n \le Q^{1/r}$ in these progressions,
we have
$$
\sum_{\substack{p^r\le Q\\ s\mid (p^r-1)}}
p^r  \ll  r^{\omega(s)}   Q(Q^{1/r} s^{-1}+1).
$$
We use this bound for $r \le \lambda$. Since
$$
\omega(s) \ll \frac{\log s}{\log \log (s+2)},
$$
for $r \le \lambda$ we have
$$
r^{\omega(s)} = \exp\( O\( \frac{\log s}{\log \log   (s+2)}\)\).
$$
The total contribution from all terms with $2 \le r \le \lambda$ is at most
\begin{equation}
\label{eq:small r}
\begin{split}
\sum_{2 \le r \le \lambda}
\sum_{\substack{p^r\le Q\\ s\mid (p^r-1)}}p^r
& \le Q(Q^{1/2} s^{-1} +1) \exp\( O\(\frac{\log s}{\log \log  (s+2)}\)\) \\
& = Q^{1+o(1)}(Q^{1/2} s^{-1} +1).
\end{split}
\end{equation}
For $\lambda \le r \le \ell$ we use the trivial bound
\begin{equation}
\label{eq:big r}
\sum_{\lambda \le r \le \ell}
\sum_{\substack{p^r\le Q\\ s\mid (p^r-1)}}p^r \le \ell Q^{1+1/\lambda}.
\end{equation}

Combining~\eqref{eq:small r} and~\eqref{eq:big r} we see that
\begin{equation}
\label{eq:sum pr}
\begin{split}
 \sum_{2 \le r \le \ell}
\sum_{\substack{p^r\le Q\\ s\mid (p^r-1)}}p^r
&\ll Q^{3/2 + o(1)} s^{-1} + Q^{1+o(1)} + Q^{1+ \varepsilon/2} \log Q\\
&\ll  Q^{3/2 + o(1)} s^{-1},
\end{split}
\end{equation}
provided that $s \le Q^{1-\varepsilon}$ and $Q \to \infty$.
Recalling~\eqref{eq q-1 q}, \eqref{eq:prime pow} and~\eqref{eq:r=1} and
that
$$
\frac{\varphi(t)}{t \varphi(s)} = \frac{1}{s},
$$
we conclude the proof for the case where $4 \nmid t$.

In the event that $4 \mid t$ then, returning to~\eqref{eq q-1 q}, we have
$$
\sum_{q \le Q}N_t(q)=\frac{\varphi(t)}{t}\sum_{\substack{q \le Q \\ s \mid (q-1)\\4 \mid (q-1)}}(q-1)
=\frac{\varphi(t)}{t}\sum_{\substack{q \le Q \\ \lcm(4,\rad(t)) \mid (q-1)}}(q-1).
$$
Since $\lcm(4,\rad(t))=2\rad(t)$, the proof now continues as before, replacing $s$ with $2s$.

\subsection{Proof of Theorem~\ref{thm:LowBound q}}

Combining~\eqref{eq q-1 q} and~\eqref{eq:prime pow},  we have
\begin{equation}
\begin{split}
\label{eq sum p upper bound 2}
\sum_{q \le Q} N_t(q) \ge \sum_{p\le Q} N_t(p)
&=\frac{\varphi(t)}{t} \sum_{\substack{p \le Q\\\rad_4(t)\mid (p-1)}}(p-1) \\
&\ge \frac{\varphi(t)}{t} \sum_{\substack{p \le Q\\ 2s\mid (p-1)}}(p-1) ,
\end{split}
\end{equation}
where, as before, $s = \rad(t)$.

It immediately follows from Lemma~\ref{lem:Linnik} that
$$
\pi(Q;2s,1)\gg  \frac{Q}{\varphi(2s)\sqrt{2s} \log Q} \ge \frac{Q}{\varphi(s)\sqrt{s} \log Q}.
$$
Thus
$$
\sum_{\substack{p \le Q\\2s\mid (p-1)}}p  \ge \sum_{k=1}^{\pi(Q; s,1)}\(2ks+1\)\ge 2s\frac{\pi(Q; s,1)^2}{2} \gg\frac{Q^2}{\varphi^2(s)(\log Q)^2}.
$$

Combining this lower bound with~\eqref{eq sum p upper bound 2} completes the proof.

\subsection{Proof of Theorem~\ref{thm:UppBound t}}

Fix any positive $T$ and $q$. For $q-1\equiv 0 \pmod 4$ we have, using
Lemma~\ref{lem:ntq},
\begin{equation}\label{eq q-1=0}
\sum_{t \le T}N_t(q)=(q-1)\sum_{\substack{t \le T \\\rad(t)|(q-1)}}\frac{\varphi(t)}{t}\le (q-1)\sum_{\substack{t \le T \\\rad(t)|(q-1)}}1.
\end{equation}
For $q-1 \not\equiv 0 \pmod 4$ we have , using Lemma ~\ref{lem:ntq},
\begin{equation}
\begin{split}
\label{eq q-1not=0}
\sum_{t \le T}N_t(q)&=(q-1)\sum_{\substack{t \le T \\\rad(t)|(q-1)\\4 \nmid t}}\frac{\varphi(t)}{t}\le (q-1)\sum_{\substack{t \le T \\\rad(t)|(q-1)}}\frac{\varphi(t)}{t}\\& \le (q-1)\sum_{\substack{t \le T \\\rad(t)|(q-1)}}1.
\end{split}
\end{equation}

Combining~\eqref{eq q-1=0}, \eqref{eq q-1not=0} and Lemma~\ref{lem:SmallDiv} completes the proof.

\subsection{Proof of Theorem~\ref{thm:Asymp q t}}

 Using~\eqref{eq q-1 q}, \eqref{eq:prime pow}
 and~\eqref{eq:sum pr} we have
 \begin{equation}
\begin{split}
 \label{double sum nmid}
\sum_{t \le T}\hskip-18 pt{\phantom{\sum}}^\sharp\, \sum_{q\le Q} N_t(q)
&=
\sum_{t \le T}\hskip-18 pt{\phantom{\sum}}^\sharp\,\frac{\varphi(t)}{t} \sum_{\substack{p\le Q\\ t\mid(p-1)}}p+
 O\(Q^{3/2+o(1)} \sum_{ t\le T} t^{-1} \) \\
 &=\sum_{t \le T}\hskip-18 pt{\phantom{\sum}}^\sharp\,\frac{\varphi(t)}{t}
 \sum_{\substack{p\le Q\\ t \mid(p-1)}}p + O\(Q^{3/2+o(1)} \),
\end{split}
\end{equation}
as $T \le Q^{1/2}$.

 Using partial summation we have
\begin{equation}
\label{eq Sum P1}
\sum_{\substack{p\le Q\\ t\mid(p-1)}}p =(Kt+1)\pi(Kt+1;t,1)-t\sum_{1 \le k \le K}\pi(kt;t,1),
\end{equation}
where $K=\fl{(Q-1)/t}$.

We now write
$$
\cE(Q,t) =   \max_{R\le Q} \left|\pi(R;t,1) - \frac{\pi(R)}{\varphi(t)}\right|.
$$

With this notation we derive from~\eqref{eq Sum P1}  that
\begin{equation}
\label{eq Sum P2}
\sum_{\substack{p\le Q\\t\mid(p-1)}}p  =
\frac{Q \pi(Q)}{\varphi(t)}
- \frac{t}{\varphi(t)} \sum_{1 \le k \le K}  \pi(kt)+O\(tK\cE(Q,t)\).
\end{equation}
By the prime number theorem and~\cite[Corollary~5.29]{IwKow}, and noting that
for $1 \le k \le K$ we have $kt\le Q$,
we also conclude that
\begin{equation*}
\begin{split}
 \sum_{1 \le k \le K}  \pi(kt) & =
t \sum_{1 \le k \le K}  \frac{k}{\log (kt)} +O(Q^2(\log Q)^{-2})\\
 & =
t \sum_{K/(\log Q)^2 \le k \le K}  \frac{k}{\log (kt)} +O(Q^2(\log Q)^{-2}).
\end{split}
\end{equation*}

Now, for $K/(\log Q)^2 \le k \le K$ we have
$$
 \frac{1}{\log (kt)} =  \frac{1}{\log Q + O(\log \log Q)}
 =   \frac{1}{\log Q} +  O\(\frac{\log \log Q}{(\log Q)^2}\).
$$
Therefore
\begin{equation*}
\begin{split}
 \sum_{1 \le k \le K}  \pi(kt) & = \(\frac{1}{2} + o(1)\)  \frac{t}{ \log Q} K^2
 = \(\frac{1}{2} + o(1)\) \frac{Q^2}{t\log Q}.
\end{split}
\end{equation*}

Substituting this in~\eqref{eq Sum P2}  and using
$\pi(Q) \sim Q/\log Q$, we obtain
$$
\sum_{\substack{p\le Q\\t\mid(p-1)}}p  =  \(\frac{1}{2} + o(1)\) \frac{Q^2}{\varphi(t) \log Q} + O\(Q\cE(Q,t) \) .
$$
Using this bound in~\eqref{double sum nmid}   yields
\begin{equation*}
\begin{split}
 \sum_{t \le T}\hskip-18 pt{\phantom{\sum}}^\sharp\,\sum_{q\le Q}N_t(q) &
 = \(\frac{1}{2} + o(1)\)
\frac{Q^2}{2 \log Q}
 \sum_{t \le T}\hskip-18 pt{\phantom{\sum}}^\sharp\,\frac{1}{t} \\
& \qquad \qquad    +
 O\(Q^{3/2+O(1)}+Q \sum_{t\le T }\cE(Q,t)\).
\end{split}
\end{equation*}
By Lemma~\ref{lem:Bomb-Vin}, with  $A=1+\varepsilon$ and $\alpha = 3/2 + \varepsilon/2$, there is some $B>0$
such that
$$
\sum_{t \le T} \cE(Q,t) \ll  Q (\log Q)^{-A} + Q^{1/2} T (\log Q)^{\alpha}
\ll  Q (\log Q)^{-1-\varepsilon/2}.
$$
Hence
\begin{equation}
\label{eq: sumTQ}
\sum_{t \le T}\hskip-18 pt{\phantom{\sum}}^\sharp\,\sum_{q\le Q}N_t(q)=
  \(\frac{1}{2} + o(1)\)\frac{Q^2}{ \log Q}
 \sum_{t \le T}\hskip-18 pt{\phantom{\sum}}^\sharp\,\frac{1}{t} +
 O\(Q (\log Q)^{-1-\varepsilon/2}\).
\end{equation}
A simple inclusion-exclusion argument leads to the asymptotic formula
\begin{equation}
\label{eq: harmonic sf}
\sum_{t \le T}\hskip-18 pt{\phantom{\sum}}^\sharp\, \frac{1}{t}=
\(\frac{1}{\zeta(2)} + o(1)\)\log T ,
\end{equation}
see~\cite{Sur} for a much more precise result.
Substituting~\eqref{eq: harmonic sf} into~\eqref{eq: sumTQ} completes the proof.

\subsection{Proof of Theorem~\ref{thm:Lower q t}}

We proceed as in the proof of Theorem~\ref{thm:Asymp q t}
but instead of~\eqref{double sum nmid} we write
\begin{equation*}
\begin{split}
\sum_{T \le t \le 2T}\hskip-18 pt{\phantom{\sum}}^\sharp\, \sum_{q\le Q} N_t(q)
& \ge \sum_{T \le t \le 2T}\hskip-18 pt{\phantom{\sum}}^\sharp\, \sum_{Q/2 \le p\le Q} N_t(p) =   \sum_{T \le t \le 2T}\hskip-18 pt{\phantom{\sum}}^\sharp\,
\frac{\varphi(t)}{t}  \sum_{\substack{Q/2 \le p\le Q \\ t\mid(p-1)}}p\\
 &\gg Q \sum_{T \le t \le 2T}\hskip-18 pt{\phantom{\sum}}^\sharp\,\frac{\varphi(t)}{t}
 \(\pi(Q;t,1) - \pi(Q/2;t,1) \).
 \end{split}
\end{equation*}
Using Lemma~\ref{lem:Mikawa} we easily conclude the proof.

\section*{Acknowledgment}

This work was  supported in part by ARC grant~DP140100118.

\end{document}